\documentclass[11pt,reqno]{amsart}

\usepackage{amsmath,amssymb,amsfonts,amsthm,mathtools}
\usepackage{amscd}
\usepackage{enumitem}
\usepackage{microtype}
\usepackage{array}
\usepackage{booktabs}
\usepackage{tikz}
\usepackage{tikz-cd}
\usetikzlibrary{arrows.meta,positioning,calc}
\usepackage[colorlinks=true,linkcolor=blue,citecolor=blue,urlcolor=blue]{hyperref}

\numberwithin{equation}{section}

\newtheorem{theorem}{Theorem}[section]
\newtheorem{proposition}[theorem]{Proposition}
\newtheorem{lemma}[theorem]{Lemma}
\newtheorem{corollary}[theorem]{Corollary}

\theoremstyle{definition}
\newtheorem{definition}[theorem]{Definition}

\theoremstyle{remark}
\newtheorem{remark}[theorem]{Remark}

\newcommand{\modu}{\mathsf{mod}}

\newcommand{\proj}{\mathsf{proj}}
\newcommand{\add}{\mathsf{add}}
\newcommand{\Sub}{\mathsf{Sub}}
\newcommand{\rad}{\operatorname{rad}}

\newcommand{\Hom}{\operatorname{Hom}}

\newcommand{\pd}{\operatorname{pd}}
\newcommand{\findim}{\operatorname{findim}}
\newcommand{\Findim}{\operatorname{Findim}}
\newcommand{\dell}{\operatorname{dell}}
\newcommand{\ddell}{\operatorname{ddell}}

\newcommand{\op}{\mathrm{op}}
\newcommand{\cG}{\mathcal G}
\newcommand{\cT}{\mathcal T}
\newcommand{\cC}{\mathcal C}
\newcommand{\cE}{\mathcal E}

\newcommand{\Om}{\Omega}
\newcommand{\stmod}{\underline{\modu}}

\title[Derived Delooping Levels of One-Point Extensions and Finitistic Dimensions]
{Derived Delooping Levels of One-Point Extensions and Finitistic Dimensions}

\author{Hanpeng Gao}
\address{School of Mathematical Sciences, Anhui University, Hefei 230601, P. R. China}
\email{hpgao@ahu.edu.cn}

\author{Dajun Liu$^{*}$}
\address{School of Mathematics and Physics, Anhui Polytechnic University, Wuhu, China}
\email{liudajun@ahpu.edu.cn}

\author{Ruomu Xu}
\address{School of Mathematical Sciences, Anhui University, Hefei 230601, P. R. China}
\email{8971638129@163.com}

\thanks{*Corresponding author}
\date{}


\begin{document}

\begin{abstract}
	Let $B$ be a finite-dimensional algebra, $M\in\modu B$, and
	$A=B[M]$ the one-point extension. We give a stable two-step extension
	criterion for the new simple $A$-module to have derived delooping level at
	most one. For a concrete algebra $B$ with a self-injective Nakayama quotient
	$R$, we determine the derived and classical delooping levels of $B[M]$ for
	every nonzero $M\in\modu R$. In particular, we obtain an explicit algebra
	$A$ satisfying
	\[
	\Findim(A^{\mathrm{op}})=1<2=\ddell A=\dell A,
	\]
	which gives a negative answer to a question of Guo and Igusa.
\end{abstract}

\maketitle

\section{Introduction}

The delooping level measures when a syzygy is a direct summand, in the stable module category, of a syzygy of higher order. G\'elinas introduced this invariant and proved that it bounds the big finitistic dimension of the opposite algebra \cite{Gelinas}. Guo and Igusa introduced the derived delooping level and established
\begin{equation}\label{eq:intro-chain}
  \Findim(\Lambda^{\op})\le \ddell\Lambda\le \dell\Lambda
\end{equation}
for every finite-dimensional algebra $\Lambda$ \cite[Theorem 2.27]{GuoIgusa}. They asked whether the first inequality is always an equality \cite[Question 5.1(2)]{GuoIgusa}. Their question is explicitly
\begin{equation}\label{eq:intro-chain1}
   \ddell\Lambda\stackrel{?}{=}\Findim(\Lambda^{\op}).
\end{equation}
The opposite algebra in this formula is essential. In particular, throughout this paper $\findim$ denotes the little finitistic dimension and $\Findim$ the big finitistic dimension.

The derived invariant allows finite exact sequences whose terms have controlled higher delooping level, and therefore detects extension data that the classical delooping level does not record. Examples with $\ddell\Lambda<\dell\Lambda$ are already known; large gaps involving the classical delooping level are studied in \cite{BarriosLanzilottaMata}, and further calculations for serial quiver algebras appear in \cite{GuoSerial}. The remaining problem is whether the left inequality in \eqref{eq:intro-chain} can be strict.

Our first step is a stable reformulation of the condition $\ddell_\Lambda X\le 1$. Put
\[
  \cG_2(\Lambda)=\add\bigl(\proj\Lambda\cup \Om_\Lambda^2(\modu\Lambda)\bigr).
\]
If $\mathcal X,\mathcal Y$ are subcategories of a module category, write $\mathcal X*\mathcal Y$ for the full subcategory of middle terms of short exact sequences with submodule in $\mathcal X$ and quotient in $\mathcal Y$. We prove
\[ \ddell_\Lambda X\le 1
 \quad\Longleftrightarrow\quad
 \Om_\Lambda X\oplus Q\in \cG_2(\Lambda)*\cG_2(\Lambda)
 \text{ for some }Q\in\proj\Lambda.\]

Now let $B$ be a finite-dimensional algebra, $M\in\modu B$, and consider the one-point extension
\[
 A=B[M]=
 \begin{pmatrix}
 k&M\\0&B
 \end{pmatrix}.
\]
Set
\[
  \cT_M=\Sub(B\oplus M),
  \qquad
  \cC_M=\add\bigl(\proj B\cup\Om_B\cT_M\bigr),
\]
where $\Sub(B\oplus M)$ denotes the submodules of finite direct sums of copies of $B\oplus M$. If $S_0$ is the new simple module, then
\[
 \ddell_A S_0\le 1
 \quad\Longleftrightarrow\quad
 M\oplus Q\in \cC_M*\cC_M
 \text{ for some }Q\in\proj B.\]
We apply this criterion to the bound quiver algebra
\begin{equation*}
 B=kQ_B/\langle bdbd,dbdb,dbdc\rangle,
 \qquad Q_B:
 \begin{tikzpicture}[baseline=-2pt,>=Stealth,scale=.9]
  \node (1) at (0,0) {$1$};
  \node (2) at (2,0) {$2.$};
  \node (3) at (-1.5,0) {$3$};
  \draw[->,bend left=18] (1) to node[above] {$b$} (2);
  \draw[->,bend left=18] (2) to node[below] {$d$} (1);
  \draw[->] (1) to node[above] {$c$} (3);
 \end{tikzpicture}
\end{equation*}
 Let $R=B/Be_3B$. Then $R$ is the self-injective Nakayama algebra on the oriented two-cycle
with radical fourth power zero. For every nonzero $M\in\modu R$, regarded
as a $B$-module, we determine the relevant homological invariants of the
one-point extension $B[M]$.

Write
\[
\mathcal E=\add(S_1\oplus S_2\oplus U_{2,2}).
\]
Our calculation gives
\[
\Findim(B[M]^{\mathrm{op}})=1,
\]
and
\[
\ddell B[M]
=
\begin{cases}
	1,& Mbdb=0,\\
	2,& Mbdb\neq 0,
\end{cases}
\qquad
\dell B[M]
=
\begin{cases}
	1,& M\in\mathcal E,\\
	2,& M\notin\mathcal E.
\end{cases}
\]
Thus, in this family, the derived delooping level is detected by the action
of the single path $bdb$, whereas the classical delooping level is
controlled by membership in $\mathcal E$.
In particular, for $M=U_{1,4}$ the path $bdb$ acts nontrivially, and the
corresponding one-point extension $A=B[U_{1,4}]=kQ/\langle bdbd,dbdb,dbdc,ac,abdc\rangle$ satisfies
\[
\Findim(A^{\mathrm{op}})
=1<2=\ddell A=\dell A,
\] where $Q$ is the following quiver:
\[ 
\begin{tikzpicture}[baseline=-2pt,>=Stealth,scale=.9]
	\node (0) at (0,-1.5) {$0$};
	\node (1) at (0,0) {$1$};
	\node (2) at (2,0) {$2$};
	\node (3) at (-1.5,0) {$3$};
	\draw[->] (0) -- node[left] {$a$} (1);
	\draw[->,bend left=18] (1) to node[above] {$b$} (2);
	\draw[->,bend left=18] (2) to node[below] {$d$} (1);
	\draw[->] (1)to node[above] {$c$} (3);
\end{tikzpicture}.\]
Hence the first inequality in \eqref{eq:intro-chain} can be strict, so the equality in
\eqref{eq:intro-chain1} does not hold in general.

The paper is organized as follows. Section~\ref{sec:prelim} recalls the delooping and
derived delooping levels and proves the stable two-step criterion.
Section~\ref{sec:onepoint} specializes this criterion to one-point extensions and studies
the opposite finitistic dimension. Section~\ref{sec:classification} carries out the explicit
calculation for the algebra $B$ above and gives the resulting families of
examples.

\section{Preliminaries}\label{sec:prelim}

Throughout, $k$ is a field and all algebras are finite-dimensional $k$-algebras. Unless specified otherwise, modules are finitely generated right modules. For an algebra $\Lambda$, let $\modu\Lambda$ and $\proj\Lambda$ denote its module category and the full subcategory of finitely generated projective modules. The stable category $\stmod\Lambda$ is the quotient by morphisms factoring through projectives. We write $\Om_\Lambda X$ for the kernel of a projective cover of $X$.

\begin{definition}[Guo--Igusa]\label{def:dell}
For $r\ge1$ and $X\in\modu\Lambda$, set
\[
 r\text{-}\dell_\Lambda X
 =\inf\bigl\{n\ge0\mid \Om_\Lambda^nX\text{ is a retract of }
 \Om_\Lambda^{n+r}Y\text{ in }\stmod\Lambda
 \text{ for some }Y\in\modu\Lambda\bigr\}.
\]
Write $\dell_\Lambda X=1\text{-}\dell_\Lambda X$. The derived delooping level $\ddell_\Lambda X$ is the infimum of the integers $m\ge0$ for which there are $0\le n\le m$ and an exact sequence
\[
 0\longrightarrow C_n\longrightarrow C_{n-1}\longrightarrow\cdots
 \longrightarrow C_0\longrightarrow X\longrightarrow0
\]
such that
\[
 (i+1)\text{-}\dell_\Lambda C_i\le m-i
 \qquad (0\le i\le n).
\]
The infimum of the empty set is $\infty$. Finally, $\dell\Lambda$ and $\ddell\Lambda$ are the maxima of the corresponding invariants over the simple right $\Lambda$-modules.
\end{definition}

We use $\findim\Lambda$ for the supremum of the finite projective dimensions of finitely generated right modules, and $\Findim\Lambda$ for the corresponding supremum over all right modules. Thus $\Findim(\Lambda^{\op})$ concerns arbitrary left $\Lambda$-modules.

For $r\ge1$, define
\[
 \cG_r(\Lambda)=\add\bigl(\proj\Lambda\cup\{\Om_\Lambda^rX\mid X\in\modu\Lambda\}\bigr).
\]
For full subcategories $\mathcal U,\mathcal V\subseteq\modu\Lambda$, let $\mathcal U*\mathcal V$ consist of those modules $E$ for which there is a short exact sequence
\[
 0\longrightarrow U\longrightarrow E\longrightarrow V\longrightarrow0,
 \qquad U\in\mathcal U,\ V\in\mathcal V.
\]

\begin{lemma}\label{lem:stable-retract}
Let $X,Y\in\modu\Lambda$. The module $X$ is a retract of $Y$ in $\stmod\Lambda$ if and only if $X$ is a direct summand of $Y\oplus P$ for some $P\in\proj\Lambda$. Consequently, for $r\ge1$ and $n\ge0$,
\begin{equation}\label{eq:r-dell-G}
 r\text{-}\dell_\Lambda X\le n
 \quad\Longleftrightarrow\quad
 \Om_\Lambda^nX\in\cG_{n+r}(\Lambda).
\end{equation}
In particular,
\begin{align*}
 r\text{-}\dell_\Lambda X=0
 &\quad\Longleftrightarrow\quad X\in\cG_r(\Lambda),\\
 \dell_\Lambda X\le1
 &\quad\Longleftrightarrow\quad \Om_\Lambda X\in\cG_2(\Lambda).
\end{align*}
\end{lemma}

\begin{proof}
Suppose that $f:X\to Y$ and $g:Y\to X$ induce a retraction in the stable category. Write
\[
 1_X-gf=ba
\]
with $a:X\to P$ and $b:P\to X$ factoring through a finitely generated projective module $P$. Then
\[
 \binom{f}{a}:X\longrightarrow Y\oplus P
\]
has left inverse $(g\ b)$. The converse follows by applying the quotient functor to a splitting.

A finite direct sum of $r$th syzygies is, up to a projective summand, the $r$th syzygy of a finite direct sum. Hence the first assertion identifies stable retracts of $r$th syzygies with $\cG_r(\Lambda)$. If $r\text{-}\dell_\Lambda X=j\le n$, applying $\Om_\Lambda^{n-j}$ in the stable category to a retraction of $\Om_\Lambda^jX$ from $\Om_\Lambda^{j+r}Y$ gives the required retraction in degree $n$. The converse is immediate from the definition. This proves \eqref{eq:r-dell-G}, and the special cases follow.
\end{proof}

\begin{lemma}\label{lem:epi-projective}
Let $p:P_X\to X$ be a projective cover and let $f:P\to X$ be an epimorphism with $P$ finitely generated projective. Then there is a projective module $Q$ and an isomorphism $P\simeq P_X\oplus Q$ under which $f=(p\ 0)$. In particular,
\[
 \ker f\simeq\Om_\Lambda X\oplus Q.
\]
\end{lemma}

\begin{proof}
Projectivity gives maps $u:P_X\to P$ and $v:P\to P_X$ such that $fu=p$ and $pv=f$. Hence $pvu=p$. Since $p$ is a projective cover, $vu$ is invertible. Replacing $u$ by $u(vu)^{-1}$ gives $vu=1$, $fu=p$, and $f=pv$. Thus
\[
 P=u(P_X)\oplus\ker v,
\]
and $f$ vanishes on $Q:=\ker v$. Taking kernels proves the final assertion.
\end{proof}

\begin{theorem}\label{thm:ddell-one}
Let $X\in\modu\Lambda$. The following conditions are equivalent.
\begin{enumerate}[label=\textup{(\arabic*)}]
\item $\ddell_\Lambda X\le1$.
\item There is an exact sequence
\[
 0\longrightarrow K\longrightarrow C\longrightarrow X\longrightarrow0
\]
with $K,\Om_\Lambda C\in\cG_2(\Lambda)$.
\item There exists $Q\in\proj\Lambda$ such that
\[
 \Om_\Lambda X\oplus Q\in\cG_2(\Lambda)*\cG_2(\Lambda).
\]
\end{enumerate}
\end{theorem}

\begin{proof}
For the bound $1$, the defining resolution for $\ddell_\Lambda X$ has length at most one. The length-zero case is included in (2) by taking $K=0$. Lemma~\ref{lem:stable-retract} therefore gives the equivalence of (1) and (2).

Assume (2), and compose a projective cover $\pi:P_C\to C$ with $C\to X$. Let
 $H$ be the kernel of this composite. Consider the commutative diagram:
\[\begin{tikzcd}
	0 \arrow[r]
	& H \arrow[r] \arrow[d, dashed]
	& P_C \arrow[r] \arrow[d]
	& X \arrow[r] \arrow[d, equal]
	& 0 \\
	0 \arrow[r]
	& K \arrow[r]
	& C \arrow[r]
	& X \arrow[r]
	& 0, 
\end{tikzcd}\]
Snake Lemma gives an exact sequence
\[
 0\longrightarrow\Om_\Lambda C\longrightarrow H\longrightarrow K\longrightarrow0.
\]
By Lemma~\ref{lem:epi-projective}, $H\simeq\Om_\Lambda X\oplus Q$ for some projective $Q$. This proves (3).

Conversely, suppose
\[
 0\longrightarrow L\xrightarrow{\iota}\Om_\Lambda X\oplus Q
 \xrightarrow{\rho}K\longrightarrow0
\]
is exact with $L,K\in\cG_2(\Lambda)$. Let
$
 j:\Om_\Lambda X\oplus Q\longrightarrow P_X\oplus Q
$
be the direct sum of the kernel inclusion with $1_Q$, and put $C=\operatorname{coker}(j\iota)$. The universal property of the cokernel gives
 an epimorphism $C\to X$, whose kernel is naturally $K$. Thus
\[
 0\longrightarrow K\longrightarrow C\longrightarrow X\longrightarrow0
\]
is exact. The kernel of $P_X\oplus Q\to C$ is isomorphic to $L$, so Lemma~\ref{lem:epi-projective} yields
$
 L\simeq\Om_\Lambda C\oplus P'
$
for a projective $P'$. Since $\cG_2(\Lambda)$ is closed under direct summands, $\Om_\Lambda C\in\cG_2(\Lambda)$. Hence (2) holds.
\end{proof}

\section{A criterion for one-point extensions}\label{sec:onepoint}

Let $B$ be a finite-dimensional $k$-algebra and $M\in\modu B$. Put
\[
 A=B[M]=\begin{pmatrix}k&M\\0&B\end{pmatrix},
 \qquad
 e_0=\begin{pmatrix}1&0\\0&0\end{pmatrix},
 \qquad e=1-e_0.
\]
We identify $\modu B$ with the full subcategory of $\modu A$ consisting of the modules $X$ satisfying $Xe_0=0$. Let $P_0=e_0A$ and $S_0=P_0/\rad P_0$. There is a canonical exact sequence
\begin{equation}\label{eq:new-simple}
 0\longrightarrow M\longrightarrow P_0\longrightarrow S_0\longrightarrow0.
\end{equation}
Define
\[
 \cT_M=\{N\in\modu B\mid N\text{ embeds into an object of }\add(B\oplus M)\},
\]
and
\[
 \cC_M=\add\bigl(\proj B\cup\{\Om_BN\mid N\in\cT_M\}\bigr).
\]

\begin{lemma}\label{lem:G2-onepoint}
With the notation above,
\begin{equation}\label{eq:G2-onepoint}
 \cG_2(A)=\add\bigl(\{P_0\}\cup\cC_M\bigr).
\end{equation}
Moreover, every module in $\cC_M$ embeds into a finitely generated projective $B$-module.
\end{lemma}

\begin{proof}
Let $N\in \operatorname{mod} B$.
	Since  $
	\operatorname{Hom}_A(P_0,N)
\simeq Ne_0
	=0,
$ we have the projective cover of $N$ as an $A$-module contains no direct
	summand isomorphic to $P_0$. Therefore its projective cover is a
	projective $B$-module. Consequently,
	\begin{equation}\label{eq:syzygy-comparison}
		\Omega_A N \simeq \Omega_B N
		\qquad
		\text{for every }N\in\operatorname{mod}B.
	\end{equation}
	
	Let $X\in\operatorname{mod}A$, and
$
	P\twoheadrightarrow X
$ be a projective cover. Since the indecomposable projective $A$-modules
	are $P_0$ together with the indecomposable projective $B$-modules, we
	may write
$
	P\simeq P_0^{\,r}\oplus P_B
$
	for some $P_B\in\operatorname{proj}B$. Then
$
	\Omega_A X\subseteq \operatorname{rad}P	\simeq
	M^{\,r}\oplus \operatorname{rad}_B P_B.
$
Since $\operatorname{rad}_B P_B$ is a submodule of the projective
	$B$-module $P_B$, it follows that $\operatorname{rad}P$ embeds into an
	object of
$
	\operatorname{add}(B\oplus M)
$
which implies
$
	\Omega_A X\in \mathcal T_M.
$
	In particular, $\Omega_A X$ is annihilated by $e_0$, so it is a
	$B$-module. Using \eqref{eq:syzygy-comparison}, we obtain
$
	\Omega_A^2X
	\simeq
	\Omega_B(\Omega_A X)
	\in \mathcal C_M.
$
	Since every projective $A$-module belongs to
$
	\operatorname{add}\bigl(\{P_0\}\cup\operatorname{proj}B\bigr)
	\subseteq
	\operatorname{add}\bigl(\{P_0\}\cup\mathcal C_M\bigr),
$ we have
$
	\mathcal G_2(A)
	\subseteq
	\operatorname{add}\bigl(\{P_0\}\cup\mathcal C_M\bigr).
$
	
	For the converse inclusion, it is enough to show that
$
	\Omega_BN\in\mathcal G_2(A)$
	\text{for every }$ N\in\mathcal T_M,
$
	since every projective $B$-module and $P_0$ are projective
	$A$-modules.
Since $N$ embeds into an object of
	$\operatorname{add}(B\oplus M)$,
	there exist  an embedding
$
	N\hookrightarrow B^{a}\oplus M^{b}
$
	for some $a,b\geq 0$.
	The canonical inclusion $M\hookrightarrow P_0$ therefore gives an
	$A$-module embedding
$
	N\hookrightarrow B^{a}\oplus P_0^{\,b}.
$
	Let $Z$ be its cokernel. We obtain an exact sequence
	\[
	0\longrightarrow N
	\longrightarrow B^{a}\oplus P_0^{\,b}
	\longrightarrow Z
	\longrightarrow 0,
	\]
	whose middle term is projective over $A$.
	By Lemma~\ref{lem:epi-projective}, there exists a projective $A$-module $Q$ such that
$
	N\simeq \Omega_A Z\oplus Q.
$
	Since $Ne_0=0$, also $Qe_0=0$. Thus $Q$ contains no direct summand
	isomorphic to $P_0$, and hence $Q$ is in fact projective over $B$.
		Taking first syzygies over $A$ yields
$
	\Omega_A N
	\simeq
	\Omega_A^2 Z,
$
	because $\Omega_A Q=0$. On the other hand,
	\eqref{eq:syzygy-comparison} gives
$
	\Omega_A N\simeq \Omega_BN.
$
	Therefore
$
	\Omega_BN
	\simeq
	\Omega_A^2Z
	\in \mathcal G_2(A).
$
	This proves
$
	\operatorname{add}\bigl(\{P_0\}\cup\mathcal C_M\bigr)
	\subseteq
	\mathcal G_2(A),
$
	and hence the desired equality.
	
Finally, every generator of $\mathcal C_M$ embeds into a finitely
generated projective $B$-module. Indeed, this is immediate for
projective $B$-modules, while for $N\in\mathcal T_M$ one has
$
\Omega_BN\subseteq P_N,
$
where $P_N\twoheadrightarrow N$ is a projective cover.
Since this property is preserved under finite direct sums and direct
summands, every object of $\mathcal C_M$ embeds into a finitely
generated projective $B$-module.
\end{proof}

\begin{lemma}\label{lem:remove-new-projective}
Let $\mathcal C\subseteq\modu B$ be closed under finite direct sums and direct summands, and put
\[
 \mathcal D=\add(\{P_0\}\cup\mathcal C).
\]
For $X\in\modu B$, the following are equivalent.
\begin{enumerate}[label=\textup{(\arabic*)}]
\item $X\oplus P\in\mathcal D*\mathcal D$ for some $P\in\proj A$.
\item $X\oplus Q\in\mathcal C*\mathcal C$ for some $Q\in\proj B$.
\end{enumerate}
\end{lemma}

\begin{proof}
The implication (2)$\Rightarrow$(1) is immediate. 

Conversely,	assume that
$
	X\oplus P\in\mathcal D*\mathcal D
$
	for some $P\in\operatorname{proj}A$. Thus there is a short exact
	sequence
	\[
	0\longrightarrow D_1
	\xrightarrow{\iota}
	X\oplus P
	\xrightarrow{\pi}
	D_2
	\longrightarrow0
	\]
	with $D_1,D_2\in\mathcal D$.
We may write
$
	P\simeq P_0^{\,b}\oplus Q
$
	for some $b\geq0$ and $Q\in\operatorname{proj}B$.
	Moreover, 
	there exist integers $a,c\geq0$ and modules $L,K\in\mathcal C$ such
	that
$
	D_1\simeq P_0^{\,a}\oplus L$ and
$
	D_2\simeq P_0^{\,c}\oplus K.
$
	Hence we have an exact sequence
	\begin{equation}\label{eq:split-P0-sequence}
		0\longrightarrow
		P_0^{\,a}\oplus L
		\longrightarrow
		P_0^{\,b}\oplus X\oplus Q
		\longrightarrow
		P_0^{\,c}\oplus K
		\longrightarrow0.
	\end{equation}
	Since $P_0^{\,c}$ is projective, the composite
$
	P_0^{\,b}\oplus X\oplus Q
	\longrightarrow
	P_0^{\,c}\oplus K
	\longrightarrow
	P_0^{\,c}
$
	is a split epimorphism. Therefore the middle term decomposes as
$
	P_0^{\,b}\oplus X\oplus Q
	\simeq
	P_0^{\,c}\oplus E
$
	for some $A$-module $E$, and after removing this split projective
	summand, \eqref{eq:split-P0-sequence} induces a short exact sequence
	\begin{equation}\label{eq:reduced-sequence}
		0\longrightarrow
		P_0^{\,a}\oplus L
		\longrightarrow
		E
		\longrightarrow
		K
		\longrightarrow0.
	\end{equation}
	Now $L$, $X$, $Q$, and $K$ are $B$-modules, so they are annihilated by
	$e_0$. In particular, none of them has a direct summand isomorphic to
	$P_0$. By the Krull--Schmidt theorem, 
$
	E\simeq P_0^{\,b-c}\oplus X\oplus Q.
$	Applying the exact functor
	$
	(-)e_0
$
	to \eqref{eq:reduced-sequence},
	we obtain
$	a=b-c.$
	Thus \eqref{eq:reduced-sequence} may be written as
	\begin{equation}\label{eq:equal-P0-sequence}
		0\longrightarrow
		P_0^{\,a}\oplus L
		\xrightarrow{\iota}
		P_0^{\,a}\oplus X\oplus Q
		\longrightarrow
		K
		\longrightarrow0.
	\end{equation}
	 Since
$
	\operatorname{Hom}_A(P_0,Y)
	\simeq Ye_0=0
$
	for every $B$-module $Y$, the map $\iota$ has the form
	\[
	\iota=
	\begin{pmatrix}
		f & h\\
		0 & g
	\end{pmatrix}
	:
	P_0^{\,a}\oplus L
	\longrightarrow
	P_0^{\,a}\oplus(X\oplus Q).
	\]
	Applying $(-)e_0$ to \eqref{eq:equal-P0-sequence} shows that
$
	f:P_0^{\,a}\longrightarrow P_0^{\,a}
$
	is an isomorphism.
		Therefore, after precomposing $\iota$ with the automorphism
$
	\begin{pmatrix}
		1 & -f^{-1}h\\
		0 & 1
	\end{pmatrix}
$
	of $P_0^{\,a}\oplus L$, and then composing on the target with the
	automorphism $f^{-1}$ on the first summand, we may replace $\iota$ by
$
	\begin{pmatrix}
		1&0\\
		0&g
	\end{pmatrix}.
$
	Since $\iota$ is injective, $g$ is injective. Moreover, the cokernel of
	$\iota$ is $K$, while the identity map on the $P_0^{\,a}$-summand
	contributes nothing to the cokernel. Hence
$
	\operatorname{coker}g\simeq K.
$
	We therefore obtain a short exact sequence
	\[
	0\longrightarrow L
	\xrightarrow{g}
	X\oplus Q
	\longrightarrow K
	\longrightarrow0.
	\]
	Since $L,K\in\mathcal C$, this shows that
$
	X\oplus Q\in\mathcal C*\mathcal C
$
	which proves $(2)$.
\end{proof}

\begin{theorem}\label{thm:onepoint-criterion}
Let $A=B[M]$. Then
\begin{equation}\label{eq:onepoint-criterion}
 \ddell_AS_0\le1
 \quad\Longleftrightarrow\quad
 M\oplus Q\in\cC_M*\cC_M
 \text{ for some }Q\in\proj B.
\end{equation}
\end{theorem}

\begin{proof}
By \eqref{eq:new-simple}, $\Om_AS_0\simeq M$. Apply Theorem~\ref{thm:ddell-one}, Lemma~\ref{lem:G2-onepoint}, and Lemma~\ref{lem:remove-new-projective} with $\mathcal C=\cC_M$.
\end{proof}

\begin{corollary}\label{cor:extension-closed}
If $\cC_M$ is extension-closed, then
\[
 \ddell_AS_0\le1\quad\Longleftrightarrow\quad M\in\cC_M.
\]
\end{corollary}

\begin{proof}
Since $0\in\cC_M$, extension-closure gives $\cC_M*\cC_M=\cC_M$. The category $\cC_M$ contains the projectives and is closed under direct summands, so $M\oplus Q\in\cC_M$ for a projective $Q$ if and only if $M\in\cC_M$.
\end{proof}

\begin{lemma}\label{lem:findim0-simple}
If $\Findim(B^{\op})=0$, then every simple right $B$-module embeds into $B$. In particular, $\dell_BS=0$ for every simple right $B$-module $S$.
\end{lemma}

\begin{proof}
Put $(-)^*=\Hom_B(-,B)$. Suppose $S^*=0$ for a simple right module $S$, and choose a minimal projective presentation
\[
 P_1\xrightarrow{f}P_0\longrightarrow S\longrightarrow0.
\]
Dualizing gives an exact sequence of finitely generated left $B$-modules
\[
 0\longrightarrow P_0^*\xrightarrow{f^*}P_1^*\longrightarrow\operatorname{coker}f^*\longrightarrow0.
\]
If this sequence split, dualizing a splitting and using the evaluation isomorphisms for finitely generated projectives would make $f$ a split epimorphism, contradicting $\operatorname{coker}f=S\ne0$. Hence $\operatorname{coker}f^*$ has projective dimension exactly one, contradicting $\Findim(B^{\op})=0$. Therefore $S^*\ne0$, so a nonzero map $S\to B$ is injective. Lemmas~\ref{lem:epi-projective} and \ref{lem:stable-retract} then give $\dell_BS=0$.
\end{proof}

\begin{proposition}\label{prop:Findim-onepoint}
	Suppose that $\Findim(B^{\op})=0$ and $M\ne0$. 	Let
	$
	A=B[M]
	$. Then
	\begin{equation}\label{eq:both-findim-one}
		\Findim(A^{\op})=1.
	\end{equation}
	Moreover, every simple right $A$-module different from  $S_0$ has delooping level zero.
\end{proposition}

\begin{proof}

	The opposite algebra is
	the triangular matrix algebra
	\[
	A^{\op}\simeq
	\begin{pmatrix}
		B^{\op} & {}_{B^{\op}}\!M_k\\
		0 & k
	\end{pmatrix}.
	\]
	By the finitistic-dimension inequality for triangular matrix rings
	\cite[Corollary~8.11]{Psaroudakis2014},
	\[
	\Findim(A^{\op})
	\le
	\Findim(B^{\op})+\Findim(k)+1.
	\]
	Therefore,
	$
	\Findim(A^{\op})\le1.
	$	It remains to prove that the value is nonzero. Put $e=1-e_0$.
	There is an exact sequence of left $A$-modules
	\[
	0\longrightarrow Ae_0Ae
	\longrightarrow Ae
	\longrightarrow Ae/Ae_0Ae
	\longrightarrow0.
	\]
	The middle term is projective, and
	$
	Ae_0Ae\simeq Ae_0\otimes_k e_0Ae
	$
	is a direct sum of copies of the projective module $Ae_0$.
	Hence
	$
	\pd_A(Ae/Ae_0Ae)\le1.
	$	We claim that the sequence does not split. Otherwise, let
	$s:Ae/Ae_0Ae\longrightarrow Ae$
	be a section. If $\bar e$ denotes the image of $e$ in the quotient,
	then
	$
	s(\bar e)=e+n
	$
	for some $n\in Ae_0Ae$. Since $M\neq0$, choose
	$0\neq m\in e_0Ae=M$. Clearly, $m\bar e=0$.
	Hence $A$-linearity of $s$ gives
	$
	0=s(m\bar e)=m(e+n)=m,
	$
	because $me=m$ and $mn=0$, the latter following from $ee_0=0$.
	This is a contradiction. Therefore
	$
	\pd_A(Ae/Ae_0Ae)=1.
	$
	Thus
	\[
	\Findim(A^{\op})=1.
	\]

	Finally,
	Lemma~\ref{lem:findim0-simple} embeds every old simple right $A$-module
	$S$ into $B=eA$. Since $eA$ is projective as a right $A$-module, $S$ is,
	up to a projective summand, a first syzygy. Lemma~\ref{lem:stable-retract}
	therefore gives $\dell_AS=0$.
\end{proof}

\begin{corollary}\label{cor:ddell-onepoint-equality}
	Suppose $\Findim(B^{\op})=0$ and $M\ne0$. Then
	\[
	\ddell A=1
	\quad\Longleftrightarrow\quad
	M\oplus Q\in\cC_M*\cC_M
	\text{ for some }Q\in\proj B.
	\]
\end{corollary}

\begin{proof}
	All old simple modules have derived delooping level zero by Proposition~\ref{prop:Findim-onepoint}. Moreover, \eqref{eq:intro-chain} and \eqref{eq:both-findim-one} give $\ddell A\ge1$. The assertion follows from Theorem~\ref{thm:onepoint-criterion}.
\end{proof}

\begin{theorem}\label{thm:ddell-one-or-two}
	Suppose that $\Findim(B^{\op})=0$, $M\ne0$, and $2\text{-}\dell_BM\le1$. Then
	\[
	\Findim(A^{\op})=1
	\qquad \text{and}\qquad
	\ddell A\in\{1,2\}.
	\]
	More precisely, $\ddell A=1$ if and only if the right-hand side of \eqref{eq:onepoint-criterion} holds. Otherwise
	\[
	\Findim(A^{\op})=1<\ddell A=\dell A=2.
	\]
\end{theorem}

\begin{proof}
	The opposite finitistic dimension is given by Proposition~\ref{prop:Findim-onepoint}, and the old simples have delooping level zero. By Lemma~\ref{lem:stable-retract}, the hypothesis $2\text{-}\dell_BM\le1$ is equivalent to
	$
	\Om_BM\in\cG_3(B).
	$
	Every projective $B$-module is projective over $A$, and \eqref{eq:syzygy-comparison} identifies the higher syzygies of $B$-modules over $B$ and $A$. Thus $\cG_3(B)\subseteq\cG_3(A)$. Since $\Om_AS_0\simeq M$, we have 
	$
	\Om_A^2S_0\simeq\Om_BM\in\cG_3(A),
	$ and 
	so $\dell_AS_0\le2$. Hence
	\[
	1=\Findim(A^{\op})\le\ddell A\le\dell A\le2.
	\]
	Theorem~\ref{thm:onepoint-criterion} decides whether the value is one; if not, both $\ddell A$ and $\dell A$ are two.
\end{proof}

\begin{remark}\label{rem:2dell}
	The condition $2\text{-}\dell_BM\le1$ is equivalent to $\Om_BM\in\cG_3(B)$. In particular it holds when $\pd_BM\le1$.
\end{remark}

\section{Examples}\label{sec:classification}

Throughout this section, let $Q$ be  the following quiver
\begin{equation*}
 \begin{tikzpicture}[baseline=-2pt,>=Stealth,scale=.9]
  \node (1) at (0,0) {$1$};
  \node (2) at (2,0) {$2$};
  \node (3) at (-1.5,0) {$3$};
  \draw[->,bend left=18] (1) to node[above] {$b$} (2);
  \draw[->,bend left=18] (2) to node[below] {$d$} (1);
  \draw[->] (1) to node[above] {$c$} (3);
 \end{tikzpicture}
\end{equation*}
and $B=kQ/I$ with $I=\langle bdbd,dbdb,dbdc\rangle$.

Let $e=e_1+e_2$ and
$R=B/Be_3B\simeq eBe.$
Then $R$ is the self-injective Nakayama algebra  with radical fourth power zero. We regard $\modu R$ as the full subcategory of $\modu B$ consisting of modules annihilated by $e_3$.

 For $i\in\{1,2\}$ and $1\le l\le4$, let $U_{i,l}$ be the uniserial $R$-module with top $S_i$ and length $l$. Put
\[
 U=U_{2,2}\simeq dbB,
 \qquad
 \cE=\add(S_1\oplus S_2\oplus U).
\]
For a nonzero $M\in\modu R$, set $A=B[M]$ and retain the notation $\cC_M$ from Section~\ref{sec:onepoint}.

\begin{lemma}\label{lem:B-syzygies}
We have
\begin{align}
 \Om_BS_1&\simeq bB\oplus S_3,& \Om_B(bB)&\simeq S_1,\label{eq:B-syz1}\\
 \Om_BS_2&\simeq dB,& \Om_B(dB)&\simeq S_2\oplus S_3,\label{eq:B-syz2}\\
 \Om_B(bdB)&\simeq bdB,& \Om_B(dbB)&\simeq dbB\oplus S_3.\label{eq:B-syz3}
\end{align}
The syzygies of the eight indecomposable $R$-modules are
\[
\begin{array}{c|cccc}
X&U_{1,1}&U_{1,2}&U_{1,3}&U_{1,4}\\ \hline
\Om_BX&bB\oplus S_3&bdB\oplus S_3&S_2\oplus S_3^{\oplus2}&S_3^{\oplus2}
\end{array}
\]
and
\[
\begin{array}{c|cccc}
X&U_{2,1}&U_{2,2}&U_{2,3}&U_{2,4}\\ \hline
\Om_BX&dB&dbB\oplus S_3&S_1\oplus S_3&S_3.
\end{array}
\]
In particular, for every $M\in\modu R$,
\[
 \Om_BM\in\cG_n(B)\qquad(n\ge1).
\]
\end{lemma}

\begin{proof}
 For a nonzero admissible path $p$ ending at $i$, the map $e_iB\to pB$, $q\mapsto pq$, is a projective cover. Its kernel is spanned by admissible paths $q$ starting at $i$ for which $pq$ belongs to the defining ideal. For $p=b,d,bd,db$, the kernels are respectively
\[
 kdbd,\qquad k\{bdb, bdc\},\qquad k\{bd, bdb, bdc\},
 \qquad k\{db, dbd, dc\}.
\]
This yields \eqref{eq:B-syz1}--\eqref{eq:B-syz3}.
The other entries in the table are obtained by the same elementary computation, and we leave the details to the reader.  Every nonprojective summand occurring in them is periodic in the stable category by \eqref{eq:B-syz1}--\eqref{eq:B-syz3}. Hence it is a stable direct summand of an $n$th syzygy for every $n$, proving the final assertion.
\end{proof}

\begin{lemma}\label{torsionless}
Let  \(X\)  be a torsionless right \(B\)-module.

\begin{enumerate}[label=\textup{(\arabic*)}]
	\item If $x\in Xe_1$, then $xc=0$ implies $xb=0$;
	\item If $y\in Xe_2$, then $ydc=0$ implies  $ydb=0$.
\end{enumerate}
	Moreover, for every nonzero \(M\in\modu R\), we have	$\mathcal C_M\cap\modu R
	=
	\mathcal E.$

\end{lemma}

\begin{proof}
	Since every torsionless module embeds into a finitely generated projective
	module, it is enough to verify (1) on the indecomposable projectives.
	Using the path bases,
	\[
	(e_1B)e_1=k\{e_1,bd\},\qquad
	(e_2B)e_1=k\{d,dbd\},
	\]
	and right multiplication by \(c\) has kernel \(0\) on the first space and
	kernel \(kdbd\) on the second. The projective \(e_3B\) has no component at vertex \(1\). Thus (1) holds for every torsionless
	\(B\)-module.
	
	Similarly, (2) can be obtained.
	
	Now let \(X\in\mathcal C_M\cap\modu R\). By Lemma~\ref{lem:G2-onepoint}, \(X\) is
	torsionless. Since \(Xe_3=0\), we have \(xc=0\) for every \(x\in Xe_1\),
	hence (1) gives \(xb=0\). Thus 
	\(X\) is a representation of the one-arrow quiver
	$
	2\xrightarrow{d}1.
$
	Therefore
$
	X\in\add(S_1\oplus S_2\oplus U)=\mathcal E,
$
	and hence
$
	\mathcal C_M\cap\modu R\subseteq\mathcal E.
$	Conversely, \(bB,dB,dbB\) are submodules of projective \(B\)-modules, so
	they belong to \(\mathcal T_M\). By Lemma~\ref{lem:B-syzygies},
	\[
	\Omega_B(bB)\simeq S_1,\qquad
	\Omega_B(dB)\simeq S_2\oplus S_3,\qquad
	\Omega_B(dbB)\simeq dbB\oplus S_3.
	\]
	Hence \(S_1,S_2,U\simeq dbB\) belong to \(\mathcal C_M\). Since they are
	all \(R\)-modules, we have
$\mathcal E\subseteq\mathcal C_M\cap\modu R.$\end{proof}

\begin{lemma}\label{4.3}
	We have
$\operatorname{Findim}(B^{\mathrm{op}})=0.$
	Consequently, 
$\operatorname{Findim}(B[M]^{\mathrm{op}})
	=
	1$ for every nonzero \(M\in\operatorname{mod}R\).
\end{lemma}

\begin{proof}Since \(B\) is finite dimensional, it is left perfect. Hence every left \(B\)-module admits a projective cover.
	Let \(Y\) be an arbitrary left \(B\)-module of finite projective
	dimension. Choose a projective cover
$
	P\twoheadrightarrow Y
$
	and put
$
	N=\ker(P\to Y).
$
	Thus \(N\subseteq \rad P\).
	Since no arrow starts at vertex \(3\), one has
$
	e_3\operatorname{rad}P=0
$
	for every projective left \(B\)-module \(P\) which implies
$e_3N=0.$
	Put \(e=e_1+e_2\), so that \(R\simeq eBe\), and let \(J_R\) denote the
	radical of \(R\). From the path bases of the indecomposable projective
	left \(B\)-modules, we have 
$
	e\operatorname{rad}(Be_i)=\operatorname{rad}(Re_i)~ (i=1,2),
$ and 
$e\operatorname{rad}(Be_3)=k\{c,dc,bdc\}.$
	All these modules are annihilated by \(J_R^3\). Hence
$J_R^3eN=0.$
Because \(e_3N=0\), the projective cover of \(N\) is  a direct sum of copies of
	\(Be_1\) and \(Be_2\). The same argument applies inductively to every
	higher syzygy of \(N\). Thus, applying the exact functor
$
	e(-):B\text{-Mod}\longrightarrow R\text{-Mod}
$
	to a finite projective resolution of \(N\) gives a finite projective
	resolution of \(eN\) over \(R\), since $
	e(Be_i)\simeq Re_i
~ (i=1,2).$
	Therefore,
	$\operatorname{pd}_R(eN)<\infty.$
Since \(R\) is self-injective, every \(R\)-module of finite projective
	dimension is projective. Hence \(eN\) is projective. But every nonzero
	projective \(R\)-module has Loewy length four, and so is not annihilated
	by \(J_R^3\).	$J_R^3eN=0$ forces
$
	eN=0.
$ Together with \(e_3N=0\) and \(1=e+e_3\), we obtain \(N=0\).
	Thus \(Y\) is projective. Consequently,
	\[
	\operatorname{Findim}(B^{\mathrm{op}})=0.
	\]
The final assertion follows immediately from Proposition \ref{prop:Findim-onepoint}.
\end{proof}

\begin{theorem}\label{thm:classification}
Let $0\ne M\in\modu R$ and $A=B[M]$. The following conditions are equivalent.
\begin{enumerate}[label=\textup{(\arabic*)}]
\item $Mbdb=0$.
\item $U_{1,4}$ is not a direct summand of $M$.
\item There is an exact sequence $0\to L\to M\to K\to0$ with $L,K\in\cC_M$.
\item For some $Q\in\proj B$, there is an exact sequence $0\to L\to M\oplus Q\to K\to0$ with $L,K\in\cC_M$.
\item $\ddell A=1$.
\end{enumerate}
Moreover,
\begin{equation}\label{eq:classification-ddell}
 \Findim(A^{\op})=1,
 \qquad
 \ddell A=
 \begin{cases}
 1,&Mbdb=0,\\
 2,&Mbdb\ne0,
 \end{cases}
\end{equation}
and
\begin{equation}\label{eq:classification-dell}
 \dell A=
 \begin{cases}
 1,&M\in\cE,\\
 2,&M\notin\cE.
 \end{cases}
\end{equation}
\end{theorem}

\begin{proof}
The indecomposable $R$-modules are precisely the $U_{i,l}$. Among them, $U_{1,4}$ is the only module on which $bdb$ acts nontrivially, so (1) and (2) are equivalent.

Every indecomposable module other than $U_{1,4}$ has a submodule $L$ and quotient $X/L$ as follows:
\begin{equation}\label{4.6}
\begin{array}{c|ccccccc}
X&U_{1,1}&U_{1,2}&U_{1,3}&U_{2,1}&U_{2,2}&U_{2,3}&U_{2,4}\\ \hline
L&0&S_2&U&0&0&S_2&U\\
X/L&S_1&S_1&S_1&S_2&U&U&U.
\end{array}
\end{equation}
All modules in the last two rows belong to $\cC_M$ by Lemma~\ref{torsionless}. Taking direct sums proves (2)$\Rightarrow$(3), and (3)$\Rightarrow$(4) is immediate.

Suppose (4) holds, and let
\[
 0\longrightarrow L\xrightarrow{\iota}M\oplus Q\xrightarrow{\pi}K\longrightarrow0,
\]
be the corresponding sequence. Let \(u\in Me_1\), viewed as an element of \(M\oplus Q\). Since
\(M\in\modu R\), we have \(Me_3=0\), and hence
$
uc=0.
$
Therefore
$
\pi(u)c=\pi(uc)=0.
$
Moreover, \(\pi(u)\in Ke_1\). Since \(K\in\mathcal C_M\), Lemma~\ref{lem:G2-onepoint}
implies that \(K\) is torsionless. Hence 
Lemma~\ref{torsionless}(1) gives
$\pi(u)b=0,
$ and hence $
\pi(ub)=0.
$
Identifying \(L\) with \(\operatorname{Im}\iota=\ker\pi\), we have \(ub\in L\).
Since \(b\) ends at vertex \(2\), we have \(ub\in Le_2\). Furthermore,
$
(ub)dc=u(bdc)\in Me_3=0.
$
Again \(L\in\mathcal C_M\), so \(L\) is torsionless. Applying Lemma~\ref{torsionless}(2) to \(ub\in Le_2\), we obtain
$
(ub)db=0.
$
Thus
$
ubdb=0
$
for every \(u\in Me_1\). Since the path \(bdb\) starts at vertex \(1\),
it follows that
$
Mbdb=Me_1bdb=0.
$ This proves \((1)\).

The old simple modules satisfy $\dell_AS_i=0$,
indeed, $S_1\simeq dbdB\subseteq e_2B$,
$S_2\simeq bdbB\subseteq e_1B$, and $S_3=e_3B$ is projective.
Let $S_0$ be the new simple module. Its first syzygy is $M$.
By Lemma~\ref{lem:B-syzygies},
$
\Omega_A^2S_0=\Omega_BM\in\mathcal G_3(B)
\subseteq\mathcal G_3(A).
$
Consequently $\dell A\leq2$. Together with
Lemma~\ref{4.3}, the usual inequalities give
\[
1=\Findim(A^{\mathrm{op}})
\leq\ddell A\leq\dell A\leq2.
\]
Theorem~\ref{thm:onepoint-criterion} identifies condition (4)
with $\ddell_AS_0\leq1$. Since the old simples
have derived delooping level zero, this proves the equivalence
with (5) and formula~\eqref{eq:classification-ddell}.

Finally, Lemma~\ref{lem:stable-retract} gives
\[
\dell_AS_0\leq1
\quad\Longleftrightarrow\quad
M\in\mathcal G_2(A)
\quad\Longleftrightarrow\quad M\in\mathcal C_M.
\]
In the last equivalence, $M$ is supported on the old vertices,
so it has no summand isomorphic to the new projective.
 Lemma \ref{torsionless} proves formula~ \eqref{eq:classification-dell}.
\end{proof}

The same argument determines the minimum number of extension factors.
\begin{definition}\label{def:stable-length}
For a module $X$, let $\lambda^{\mathrm{st}}_{\cC_M}(X)$ be the least integer $r\ge1$ such that, for some projective $B$-module $Q$, there is a filtration
\[
 0=X_0\subseteq X_1\subseteq\cdots\subseteq X_r=X\oplus Q,
 \qquad X_j/X_{j-1}\in\cC_M.
\]
If no such filtration exists, put $\lambda^{\mathrm{st}}_{\cC_M}(X)=\infty$.
\end{definition}

\begin{corollary}\label{cor:stable-length}
For $0\ne M\in\modu R$,
\[
 \lambda^{\mathrm{st}}_{\cC_M}(M)=
 \begin{cases}
 1,&M\in\cE,\\
 2,&M\notin\cE\text{ and }Mbdb=0,\\
 3,&Mbdb\ne0.
 \end{cases}
\]
\end{corollary}

\begin{proof}
Since $\mathcal C_M$ is closed under direct summands,
$M\oplus Q\in\mathcal C_M$ implies $M\in\mathcal C_M$.
Thus the value is one precisely when $M\in\mathcal E$ by
Lemma~\ref{torsionless}. Theorem~\ref{thm:classification}
settles the case of two factors and excludes two factors
when $Mbdb\ne0$. It remains to construct three factors.
For $V=U_{1,4}$, take
\[
0\subseteq\rad^3V
\subseteq\rad V\subseteq V.
\]
The factors, from bottom to top, are $S_2,U,S_1$.
Combine these filtrations with
\eqref{4.6} for the remaining
indecomposable summands of $M$. Repeated terms may be inserted
to give every summand a filtration with three factors.
\end{proof}

\begin{corollary}\label{cor:family-Al}
Let  $A_l=B[U_{1,l}]$ and $A'_l=B[U_{2,l}]$~ $(1\le l\le4)$. Then we have
\[
\begin{array}{c|cccc}
 l&1&2&3&4\\ \hline
 \Findim(A_l^{\op})&1&1&1&1\\
 \ddell A_l&1&1&1&2\\
 \dell A_l&1&2&2&2
\end{array}
\]
\[
\begin{array}{c|cccc}
	l & 1 & 2 & 3 & 4\\ \hline
	\operatorname{Findim}((A'_l)^{\mathrm{op}})
	&1&1&1&1\\
	\operatorname{ddell}A'_l
	&1&1&1&1\\
	\operatorname{dell}A'_l
	&1&1&2&2
\end{array}.
\]

\end{corollary}

\begin{proof}
It follows from Theorem \ref{thm:classification}.
\end{proof}


\section*{Acknowledgments}

Hanpeng Gao is supported by  National Natural Science Foundation of China (No.12301041),
Dajun  Liu  is supported by  National Natural Science Foundation of China (No.12101003) and
the Natural Science Foundation of Anhui province (No.2108085QA07).


\end{document}